\documentclass[a4paper]{amsart}

\usepackage{amsmath,amstext,amssymb,mathrsfs,amscd,amsthm,indentfirst, bbm}
\usepackage{amsfonts,nicefrac}
\usepackage{mathtools}

\usepackage{tikz, tikz-cd}
\usetikzlibrary{matrix,calc,positioning,arrows,decorations.pathreplacing,decorations.markings,patterns}
\tikzset{->-/.style={decoration={
			markings,
			mark=at position #1 with {\arrow{>}}},postaction={decorate}}}
\tikzset{-<-/.style={decoration={
					markings,
					mark=at position #1 with {\arrow{<}}},postaction={decorate}}}

\usepackage{booktabs}
\usepackage{verbatim}
\usepackage{enumerate}

\usepackage{graphicx}

\newcommand{\bR}{\mathbb{R}}

\newcommand{\bZ}{\mathbb{Z}}

\newcommand\lra{\longrightarrow}

\newcommand\Diff{\mathrm{Diff}}
\newcommand\Homeo{\mathrm{Homeo}}
\newcommand\Emb{\mathrm{Emb}}

\newcommand\holim{\operatorname*{holim}}
\newcommand\tohofib{\operatorname*{tohofib}}

\newcommand{\map}{\mathrm{map}}

\renewcommand{\epsilon}{\varepsilon}

\newcommand{\OO}{\mathrm{O}}

\newcommand{\Top}{\mathrm{Top}}

\newcommand{\inj}{\mathrm{inj}}
\newcommand{\aug}{\mathrm{aug}}
\newcommand{\Wh}{\mathrm{Wh}}

\mathchardef\ordinarycolon\mathcode`\:
\mathcode`\:=\string"8000
\begingroup \catcode`\:=\active
  \gdef:{\mathrel{\mathop\ordinarycolon}}
\endgroup

\usepackage{amsthm}
\theoremstyle{plain}
\newtheorem{MainThm}{Theorem}

\newtheorem{theorem}{Theorem}[section]

\newtheorem{proposition}[theorem]{Proposition}
\newtheorem{lemma}[theorem]{Lemma}

\newtheorem{corollary}[theorem]{Corollary}

\theoremstyle{definition}
\newtheorem{definition}[theorem]{Definition}

\newtheorem{example}[theorem]{Example}

\newtheorem{construction}[theorem]{Construction}

\theoremstyle{remark}

\newtheorem{remark}[theorem]{Remark}
\newtheorem*{remark*}{Remark}

\numberwithin{equation}{section}

\usepackage[hypertexnames=false]{hyperref}

\title{The Alexander trick for homology spheres}

\author{S{\o}ren Galatius}
\email{galatius@math.ku.dk}
\address{Department of Mathematics\\
	University of Copenhagen\\
	Denmark}

\author{Oscar Randal-Williams}
\email{o.randal-williams@dpmms.cam.ac.uk}
\address{Centre for Mathematical Sciences\\
Wilberforce Road\\
Cambridge CB3 0WB\\
UK}
\date{\today}

\begin{document}
\begin{abstract}
We show that the group of homeomorphisms of a compact contractible $d$-manifold which fix the boundary is contractible, as long as $d \geq 6$. We deduce this from a strong uniqueness statement for one-sided $h$-cobordisms.
\end{abstract}
\maketitle

\section{Introduction}

\subsection{Contractible manifolds}
Let $\Delta$ be a $d$-dimensional compact topological manifold, which is contractible.  Then $\partial \Delta$ is necessarily a homology sphere, but need not be simply connected.  The main result we wish to explain is as follows.

\begin{MainThm}\label{thm:A}
If $d \geq 6$ then the topological group $\mathrm{Homeo}_\partial(\Delta)$ is weakly contractible.
\end{MainThm}

In the case that $\Delta = D^d$ is the disc of any dimension, this theorem was proved a hundred years ago by J.\ W.\ Alexander \cite{Alexander}, using the explicit radial deformation now known as the ``Alexander trick''. For a general contractible manifold there is no such convenient coordinate system and we will have to proceed less directly. Nonetheless, we consider this theorem as the replacement of the Alexander trick for general $\Delta$, even though it does not provide an explicit contraction. A further interpretation of this result arises by considering the homotopy fibre sequence
\begin{equation}\label{eq:4}
\mathrm{Homeo}_\partial(\Delta) \lra \mathrm{Homeo}(\Delta) \overset{\mathrm{res}}\lra \mathrm{Homeo}(\partial \Delta)
\end{equation}
given by restricting homeomorphisms of $\Delta$ to its boundary.   It
can be seen that $\mathrm{res}$ is surjective, and Theorem~\ref{thm:A} can be interpreted as saying that \emph{families} of homeomorphisms of $\partial \Delta$ can be extended to homeomorphisms of $\Delta$, in an essentially unique way.

In the range of dimensions $d \geq 6$ which we are considering, any homology $(d-1)$-sphere $\Sigma$ is the boundary of a contractible manifold by a well-known theorem of Hsiang--Hsiang \cite[Theorem 5.6]{HsiangHsiang} and Kervaire \cite[Theorem 3]{Kervaire} (together with some triangulation theory). It is also known, although perhaps less so, that any two such bounding manifolds are homeomorphic relative to $\Sigma$. If the latter statement is interpreted as the moduli space of contractible fillings of the homology sphere $\Sigma$ being path-connected, then our theorem can be interpreted as saying that this moduli space is in fact contractible: this is the ``Alexander trick'' for homology spheres.

If $d \leq 3$ then a contractible compact manifold $\Delta$ is homeomorphic to $D^d$ (for $d=3$ by the Poincar\'e conjecture) and so the conclusion of the theorem is true by the Alexander trick. We were in fact motivated by the case $d=4$ in relation to \cite[Section 4.3]{GRWPont}, but the methods we will discuss here do not apply in that case, nor to $d=5$.

As part of the proof of Theorem \ref{thm:A} we will also prove the following statement about the topological group of diffeomorphisms of $\Delta$ when this manifold has a smooth structure.

\begin{MainThm}\label{thm:Asmooth}
  For any smooth structure on $\Delta^d$ and any smooth embedding $D^d \subset \Delta$, the map $\mathrm{Diff}_\partial(D^d) \to \mathrm{Diff}_\partial(\Delta)$ given by extending diffeomorphisms by the identity is a weak equivalence, provided $d \geq 6$.
\end{MainThm}

\subsection{One-sided $h$-cobordisms}

We will deduce Theorems \ref{thm:A} and \ref{thm:Asmooth} from a general result about smooth embeddings of one-sided $h$-cobordisms, which also go under the name of ``semi-$h$-cobordisms'', cf.\ \cite[Section 11.2]{FreedmanQuinn}, \cite[Section 3]{GuilbaultTinsleyIII}, \cite[Section 4]{GuilbaultTinsley}, \cite[Section 3]{SuYe}.

\begin{definition}
Let $\pi$ be a discrete group, $B$ be a smooth compact manifold with boundary, and $f_B : B \to B\pi$ be a 1-connected map. A \emph{one-sided $h$-cobordism} on $B$ over $B\pi$ is a smooth cobordism $C : B \leadsto B'$, restricting to a trivial cobordism $\partial B \leadsto \partial B'$, and a map $f : C \to B\pi$ extending $f_B$, such that
\begin{enumerate}[(i)]
\item $f$ is 2-connected, and 
\item $B' \to C$ is a homotopy equivalence.
\end{enumerate}
\end{definition}

In this situation it follows from Lefschetz duality that $H_*(C, B ; \bZ[\pi])=0$, i.e.\ that the inclusion $B \to C$ is an acyclic map (recall that a map $f: X \to Y$ is acyclic when its homotopy fibres have the $\bZ$-homology of a point, see also \cite{RaptisAcyclic} and the references therein for equivalent definitions).  In particular $B$ is path connected, and we choose a basepoint $b \in B$ and look at the induced homomorphism $(f_B)_*: \pi_1(B,b) \to \pi_1(B\pi,f_B(b)) \cong \pi$, whose kernel $P \subset \pi_1(B,b)$ must be perfect if a one-sided $h$-cobordism over $B\pi$ exists. Furthermore, choosing a handle structure on $C$ and eliminating 0- and 1-handles relative to $B$ (using that the inclusion $B \to C$ is 1-connected), the finitely-many 2-handles of $C$ show that $P$ is generated by finitely-many conjugacy classes in $\pi_1(B,b)$.

This has a converse by implementing the plus-construction with manifolds, as first introduced by Kervaire \cite{Kervaire}. Given a $(d-1)$-manifold $B$ and a 1-connected map $f_B : B \to B\pi$ whose kernel on $\pi_1$ is a perfect group $P$ which is generated by finitely-many conjugacy classes in $\pi_1(B,b)$, using that $d \geq 6$ one can attach 2- and 3-handles to $B$ as dictated by the plus-construction to obtain a cobordism $C^+ : B \leadsto B^+$ such that the inclusion $B \to C^+$ is acyclic and the kernel on $\pi_1$ is $P$. The map $f_B$ therefore extends to a 2-connected map $f^+ : C^+ \to B\pi$. As $C^+$ can be obtained from $B^+$ by attaching handles of index $d-2, d-3 \geq 3$, the inclusion $B^+ \to C^+$ is 2-connected, so by an application of Lefschetz duality it is an equivalence.  (A different kind of converse has been considered by \cite{Rolland}.)

An appropriate uniqueness theorem for one-sided $h$-cobordisms has been discussed by Freedman--Quinn \cite[p.\ 197]{FreedmanQuinn} and then by Guilbault--Tinsley \cite[Theorem 4.1]{GuilbaultTinsley} and most generally by Su--Ye \cite[Theorem 3.2]{SuYe}. The homotopy equivalence $B' \to C$ has a Whitehead torsion $\tau(C, B') \in \Wh(\pi)$, and these authors show that one-sided $h$-cobordisms on $B$ over $B\pi$ are classified up to diffeomorphism relative to $B$ by this torsion, at least when $\dim(B) \geq 6$. By similar considerations, Guilbault--Tinsley \cite[Theorem 3.2]{GuilbaultTinsleyIII} show that a one-sided $h$-cobordism embeds relative to $B$ into any manifold $M$, extending a given $e_0 : B \hookrightarrow \partial M$, as long $\pi_1(B,b) \to \pi_1(\partial M,b) \to \pi_1(M,b)$ factors through $(f_B)_* : \pi_1(B,b) \to \pi$. The following can be considered as a strong uniqueness statement for such embeddings.

\begin{MainThm}\label{thm:B}
Let $d \geq 6$ and $M$ be a $d$-dimensional manifold with boundary. Let $(C : B \leadsto B', f)$ be a $d$-dimensional one-sided $h$-cobordism over $B\pi$ and $e_0 : B \hookrightarrow \partial M$ be an embedding such that $\pi_1(B,b) \to \pi_1(\partial M,b) \to \pi_1(M,b)$ factors through $(f_B)_* : \pi_1(B,b) \to \pi$. Then the space $\mathrm{Emb}_{B}(C, M)$, of smooth embeddings of $C$ into $M$ relative to $B$, is contractible.
\end{MainThm}

It will be a consequence of our argument that this space of embeddings is non-empty, giving a new proof of \cite[Theorem 3.2]{GuilbaultTinsleyIII} and \cite[Theorem 4.2]{GuilbaultTinsley}. (We have not tried to produce higher-homotopy strengthenings of \cite[Theorems 5.2 and 5.3]{GuilbaultTinsleyIII} though.)
Applied with $M$ being another one-sided $h$-cobordism on $B$ over $B\pi$, it re-proves that one-sided $h$-cobordisms embed into each other, so with the $s$-cobordism theorem re-proves that one-sided $h$-cobordisms are classified by their Whitehead torsion.

We offer two proofs of Theorem \ref{thm:B}. The first proof uses Goodwillie--Weiss' embedding calculus \cite{EmbImmI,EmbImmII}. The handle dimension of a one-sided $h$-cobordism $C$ is such that this calculus converges, so $\mathrm{Emb}_{B}(C, M)$ can be accessed by direct calculation. The second proof does not rely on any machinery, and proceeds by constructing a semi-simplicial resolution of $\mathrm{Emb}_{B}(C, M)$ by subtracting handles, reminiscent of our earlier work on parameterised surgery (particularly \cite{grwstab2}). In Section \ref{sec:Consequences} we describe some consequences of Theorem \ref{thm:B} for more general spaces of embeddings, and for embedding calculus.

\begin{remark}
  In a very recent preprint, Krannich and Kupers explain how to prove our Theorems A--C for $d \geq 5$, see \cite[Theorem 6.18]{KK24}.  Part of their work establishes convergence of embedding calculus in the presence of certain handles of codimension 2, which allows an adapted version of our ``first proof'' to apply.

  It was shown recently in \cite[Theorem 1.7]{Krushkaletal} that our Theorem~\ref{thm:Asmooth} does not hold for $d=4$. Explicit $\Delta$'s for which $\pi_0(\mathrm{Diff}_\partial(D^4)) \to \pi_0(\mathrm{Diff}_\partial(\Delta))$ is not surjective are now known \cite[Theorem D]{Konnoetal}.
\end{remark}

\subsection{Proof of Theorems \ref{thm:A} and \ref{thm:Asmooth}, using Theorem~\ref{thm:B}}

Suppose first that $\Delta$ is a contractible manifold of dimension $d \geq 6$ which is endowed with a smooth structure, and choose a smooth embedding
$$ D^d \subset \mathrm{int}(\Delta).$$
We obtain a smooth cobordism $C := \Delta \setminus \mathrm{int}(D^d) : \partial \Delta \leadsto \partial D^d$, which is a one-sided $h$-cobordism on $\partial \Delta$ over $B\{e\}$. The parameterised isotopy extension theorem gives a fibration sequence
$$\Diff_{\partial }(D^d) \lra \Diff_{\partial }(\Delta) \lra \mathrm{Emb}_{\partial \Delta}(C, \Delta),$$
and by an application of Theorem \ref{thm:B} the base of this fibration is contractible and so the left-hand map is a weak equivalence: this proves Theorem \ref{thm:Asmooth}.

In Theorem \ref{thm:A} we assume merely that $\Delta$ is a contractible topological manifold of dimension $d \geq 6$, but it follows from smoothing theory \cite[Essay IV \S 10]{kirbysiebenmann} that it admits a smooth structure: by op.cit., it suffices to give a vector bundle reduction of its stable tangent microbundle and, $\Delta$ being contractible, there is no obstruction to doing so.  Choose one such smooth structure $s$ and write $s_\Sigma$ for the induced smooth structure on $\Sigma = \partial \Delta$.  We will keep the smooth structure $s_\Sigma$ fixed, but must discuss other smooth contractible manifolds bounding $(\Sigma,s_\Sigma)$ besides $(\Delta,s)$.  If $(\Delta',s')$ is another smooth contractible filling, then $s$ and $s'$ glue to a smooth structure on the $d$-dimensional homotopy sphere $\Delta \cup_\Sigma \Delta'$, and since oriented homotopy spheres up to oriented diffeomorphism form a group $\Theta_d$ with respect to connected sum, there exists a unique element $[\Gamma] \in \Theta_d$ such that $(\Delta,s) \cup_\Sigma ((\Delta',s') \# \Gamma)$ is oriented diffeomorphic to the standard $S^d$.  Then $D^{d+1}$ can be interpreted as a smooth $h$-cobordism from $(\Delta,s)$ to $(\Delta',s') \# \Gamma$ relative to $\Sigma$, so by the $h$-cobordism theorem they are diffeomorphic relative to $\Sigma$.  We have argued that the map of sets
\begin{align*}
  \Theta_d & \lra \{(\Delta',s') \mid (\partial \Delta',s'_{\vert \Sigma}) = (\Sigma,s_\Sigma) \} / \text{diffeomorphism relative to $\Sigma$}\\
  [\Gamma] & \longmapsto (\Delta,s) \# \Gamma
\end{align*}
is a bijection.

Incidentally, the same argument applies\footnote{In the topological case this uniqueness statement holds for any $d$, see for instance \cite[Remark 21.2]{DET} for the case $d=4$.} to fillings of $\Sigma$ by contractible topological manifolds, using the topological Poincar\'e conjecture \cite{NewmanEngulfing} to deduce that $\Delta \cup_\Sigma \Delta'$ is homeomorphic to $S^d$ and the topological $h$-cobordism theorem \cite[Essay III \S 3.4]{kirbysiebenmann} to deduce that $\Delta$ is homeomorphic to $\Delta'$ relative to $\Sigma$.  This uniqueness of contractible topological fillings also implies
that the map $\mathrm{res}$ in~(\ref{eq:4}) is surjective: indeed, for any $\phi \in \mathrm{Homeo}(\partial \Delta)$ we can regard $(\Delta,\phi)$ as a new filling of $\Sigma = \partial \Delta$, which must then be homeomorphic to $(\Delta,\mathrm{id})$ relative to $\Sigma$.

The moduli space of smooth contractible fillings of $(\Sigma,s_\Sigma)$ can be defined as the classifying space of the topological groupoid whose objects are the smooth structures $s'$ on $\Delta$ agreeing with $s_\Sigma$ on the boundary, and whose morphisms $s' \to s''$ are the diffeomorphisms $(\Delta,s') \to (\Delta,s'')$ relative to $\Sigma$.  This moduli space has the homotopy type
\begin{equation*}
  \coprod_{[\Gamma] \in \Theta_d} B\Diff_\partial((\Delta,s)\# \Gamma)
\end{equation*}
and comes with an evident map to $B\Homeo_\partial(\Delta)$, whose homotopy fibre can be described using the parameterised smoothing theory of \cite[Essay V]{kirbysiebenmann}.  Indeed, the relative statement of \cite[Essay V, \S 3]{kirbysiebenmann} identifies this homotopy fibre with a space of sections over $\Delta$, fixed over $\Sigma = \partial \Delta$, of a fibration with fibre $\Top(d)/\OO(d)$.  Since $\Delta$ is contractible, this fibration is trivial and we deduce a fibration sequence
\begin{equation*}
  \mathrm{map}_*(\Delta/\partial \Delta, \Top(d)/\OO(d)) \lra
  \coprod_{[\Gamma] \in \Theta_d} B\Diff_\partial((\Delta,s)\# \Gamma)
  \lra B\Homeo_\partial(\Delta).  
\end{equation*}
Comparing this fibre sequence to the analogous sequence with $D^d$ in place of $\Delta$, we deduce a square
\begin{equation*}
  \begin{tikzcd}
    \coprod\limits_{[\Gamma] \in \Theta_d} B\mathrm{Diff}_{\partial }(D^d \# \Gamma)\rar \dar& B\mathrm{Homeo}_{\partial }(D^d) \dar \\
    \coprod\limits_{[\Gamma] \in \Theta_d} B\mathrm{Diff}_{\partial }((\Delta,s)\# \Gamma)\rar& B\mathrm{Homeo}_{\partial }(\Delta)
  \end{tikzcd}
\end{equation*}
whose vertical maps are induced by gluing a one-sided $h$-cobordism $C: (\Sigma,s_\Sigma) \leadsto \partial D^d$.  Collapsing $C$ induces a weak equivalence $\Delta/\Sigma \to D^d/\partial D^d$, so the induced map of horizontal fibres is a weak equivalence by the Kirby--Siebenmann smoothing theory descriptions, and hence the square is homotopy cartesian.  Theorem~\ref{thm:Asmooth} asserts that the left vertical map in the square is a weak equivalence when restricted to a map between the components corresponding to $0 \in \Theta_d$, and the map between the components corresponding to $[\Gamma] \in \Theta_d$ may be identified with the map between the components corresponding to $0 \in \Theta_d$ by boundary connected sum with $D^d \# \Gamma$.  Since the left vertical map evidently induces a bijection on $\pi_0$, it is therefore a weak equivalence and we deduce that the right vertical map is too, finishing the proof of Theorem~\ref{thm:A}.

\subsection{Acknowledgements}
SG was supported by the Danish National Research Foundation (DNRF151). ORW was supported by the ERC under the European Union's Horizon 2020 research and innovation programme (grant agreement No.\ 756444).

\section{First proof of Theorem \ref{thm:B}: Embedding calculus}\label{sec:EmbCalcProof}

We can approach the space of embeddings $\mathrm{Emb}_{B}(C, M)$ using the embedding calculus of Goodwillie and Weiss \cite{EmbImmI,EmbImmII}. Recall that this theory provides a tower
$$\cdots\lra T_3\mathrm{Emb}_{B}(C, M) \lra T_2\mathrm{Emb}_{B}(C, M) \lra T_1\mathrm{Emb}_{B}(C, M)$$
of spaces under $\mathrm{Emb}_{B}(C, M)$, and hence a map
\begin{equation*}
\mathrm{Emb}_{B}(C, M) \lra T_\infty \mathrm{Emb}_{B}(C, M) := \holim_r T_r \mathrm{Emb}_{B}(C, M).
\end{equation*}
We will first prove that $T_\infty \mathrm{Emb}_{B}(C, M) \simeq *$. To do so it suffices to show that $T_1\mathrm{Emb}_{B}(C, M)$ and the homotopy fibres
$$L_k \mathrm{Emb}_{B}(C, M) := \mathrm{hofib}_{e}(T_k\mathrm{Emb}_{B}(C, M) \to T_{k-1}\mathrm{Emb}_{B}(C, M)),$$
taken at any $e \in T_{k-1}\mathrm{Emb}_{B}(C, M)$ whose underlying map $T_0e : C \to M$ is over $B\pi$, are all contractible. 

In order to do so, we will repeatedly use the following lemma.  The first part of the lemma is formally a special case of the second, but we prefer to state and prove it separately.  It can be regarded as a higher-categorical universal property of acyclic maps (cf.\ \cite[Lemma 3.4]{RaptisAcyclic}).

\begin{proposition}\label{prop:UPAcyclic}\mbox{}
\begin{enumerate}[(i)]
\item\label{it:UPAcyclic1} If $X \hookrightarrow Y$ is a cofibration which is acyclic, then for any map $g : X \to Z$ the space $\mathrm{map}_X(Y, Z)$ of maps $f : Y \to Z$ extending $g$ is either empty or contractible; it is non-empty if and only if $g_* : \pi_1(X, x_0) \to \pi_1(Z, g(x_0))$ factors over $\pi_1(X, x_0) \to \pi_1(Y, x_0)$ for each basepoint $x_0 \in X$.
\item\label{it:UPAcyclic2} If $X \hookrightarrow Y$ is a cofibration which is acyclic, $p : E \to Y$ is a fibration, and $g : X \to E$ is a section of $p$ over $X$, then the space $\Gamma_X(p : E \to Y)$ of sections $f : Y \to E$ of $p$ extending $g$ is either empty or contractible; it is non-empty if and only if $g_* : \pi_1(X, x_0) \to \pi_1(E, g(x_0))$ factors over $\pi_1(X, x_0) \to \pi_1(Y, x_0)$ for each basepoint $x_0 \in X$.
\end{enumerate}
\end{proposition}
\begin{proof}
For part (i), we may suppose that $X$, $Y$, and $Z$ are path-connected. Using the given assumption on fundamental groups, the classical universal property of acyclic maps \cite[Proposition 3.1]{HausmannHusemoller} says that the space $\mathrm{map}_X(Y, Z)$ is non-empty and path-connected.

Modelling the Postnikov truncation $Z \to B\pi_1(Z,g(x_0))$ as a fibration, note that the map $\mathrm{map}_{X, /B\pi_1(Z,g(x_0))}(Y, Z) \overset{\sim}\to \mathrm{map}_{X}(Y, Z)$ from the space of maps over $B\pi_1(Z,g(x_0))$ is an equivalence, because $\mathrm{map}_{X}(Y, B\pi_1(Z,g(x_0)))$ is easily seen to be contractible (again using the assumption on fundamental groups). The Moore--Postnikov tower of the map $Z \to B\pi_1(Z,g(x_0))$ yields a spectral sequence
$$E^2_{s,t} = H^s(Y, X ; \underline{\pi_t(Z,g(x_0))}) \Rightarrow \pi_{t-s}(\mathrm{map}_{X, /B\pi_1(Z,g(x_0))}(Y, Z), f), \quad t > 1,$$
whose $E^2$-page is identically zero as $X \hookrightarrow Y$ is acyclic. This shows that the space $\mathrm{map}_{X, /B\pi_1(Z,g(x_0))}(Y, Z) \simeq \mathrm{map}_{X}(Y, Z)$ is contractible.

For part (ii), there is a homotopy fibre sequence
$$\Gamma_X(p : E \to Y) \lra \mathrm{map}_X(Y, E) \lra \mathrm{map}_X(Y, Y),$$
where the fibre is taken over the identity map of $Y$. An application of part (i) shows that the base is contractible, and a further application of part (i) shows that the total space is contractible under the given assumption on fundamental groups. Hence $\Gamma_X(p : E \to Y)$ is contractible under this assumption.
\end{proof}

\begin{lemma}
$T_1\mathrm{Emb}_{B}(C, M)$ is contractible.
\end{lemma}
\begin{proof}
  $T_1\mathrm{Emb}_{B}(C, M)$ is equivalent to the space of vector bundle maps $TC\to TM$, i.e., continuous maps $C \to M$ covered by fibrewise linear isomorphisms, which are the identity over $B$. This may be described as sections of the bundle $p : \mathrm{Iso}(TC, TM) \to C$ whose fibre over $c \in C$ is the space of linear isomorphisms from $T_c C$ to some $T_m M$, which extend a given section $g: B \to \mathrm{Iso}(TC, TM)$, i.e.\ as $\Gamma_B(p : \mathrm{Iso}(TC, TM) \to C)$. By Proposition \ref{prop:UPAcyclic} (\ref{it:UPAcyclic2}) this space is contractible as long as $g_*: \pi_1(B, b) \to \pi_1(\mathrm{Iso}(TC, TM), g(b))$ factors over the quotient $\pi_1(B, b) \to \pi_1(C, b) = \pi$. The kernel of the latter is a perfect normal subgroup $P \lhd \pi_1(B, b)$. In the diagram
\begin{equation*}
  \begin{tikzcd}
& \pi_1(B, b) \dar{g_*} \\
\pi_1(\mathrm{GL}_d(\bR), \mathrm{id})\rar & \pi_1(\mathrm{Iso}(TC, TM), g(b)) \rar{p_*} & \pi_1(M, b),
  \end{tikzcd}
\end{equation*}
the assumption on fundamental groups in Theorem \ref{thm:B} says that $P$ is in the kernel of $p_* \circ g_*$, so $g_*(P)$ lies in the kernel of $p_*$. But this kernel is a quotient of $\pi_1(\mathrm{GL}_d(\bR), \mathrm{id}) \cong \bZ/2$ so an abelian group,
and $g_*(P)$ is perfect and so must be trivial. Thus $g_*$ does indeed factor over $\pi_1(B, b) \to \pi_1(C, b) = \pi$.
\end{proof}

For later use, we point out a forgetful map
\begin{equation}\label{eq:ImmToMaps}
T_1\mathrm{Emb}_{B}(C, M) \lra \mathrm{map}_B(C, M),
\end{equation}
which in the notation of the above proof is given by forgetting the fibrewise linear isomorphisms.

\begin{lemma}
$L_k \mathrm{Emb}_{B}(C, M)$ is contractible for each $k \geq 2$.
\end{lemma}
\begin{proof}
  As in \cite[Sections 9 and 10]{EmbImmI} the higher layers $L_k \mathrm{Emb}_{B}(C, M)$ in the embedding calculus tower can be described as spaces of sections of certain fibrations
$$\pi_k : E_k \lra C_k(C),$$
where these sections are prescribed near the diagonals, and near the locus where some configuration point lies in $B \subset C$. More precisely, let $\tilde{C}_k(C) \subset C^k$ be the ordered configuration space of $k$ points in $C$, so that $C_k(C) = \tilde{C}_k(C)/\mathfrak{S}_k$, let
\begin{equation}
D_k := \{(x_1, \ldots, x_k) \in C^k \, | \, x_i = x_j \text{ for some } i \neq j, \text{ or } x_i \in B \text{ for some } i\},\label{eq:1}
\end{equation}
and let $U \subset C^k$ be a $\mathfrak{S}_k$-invariant regular neighbourhood of $D_k$. Then write $\tilde{\nabla}_\partial = \tilde{C}_k(C) \cap U \subset \tilde{C}_k(C)$ and  $\nabla_\partial := \tilde{\nabla}_\partial/\mathfrak{S}_k \subset \tilde{C}_k(C)/\mathfrak{S}_k = C_k(C)$. Then the space in question will be the space $\Gamma_{\nabla_\partial}(\pi_k : E_k \to C_k(C))$ of sections of $\pi_k$ which agree with a prescribed section on $\nabla_\partial$.  

We wish to apply Proposition \ref{prop:UPAcyclic} (\ref{it:UPAcyclic2}) to this space of sections. In Lemma \ref{lem:acyclic} below we will show that the inclusion $D_k \to C^k$ is acyclic, so $U \to C^k$ is too. As the inclusion $\tilde{C}_k(C) \subset C^k$ is $(d-1)$-connected, and so in particular a $\pi_1$-isomorphism, it follows by excision that the inclusion $\tilde{\nabla}_\partial \to \tilde{C}_k(C)$ is acyclic. Taking (homotopy) orbits of both spaces by $\mathfrak{S}_k$ does not change the homotopy fibres, so the inclusion ${\nabla}_\partial \to {C}_k(C)$ is also acyclic.

To apply Proposition \ref{prop:UPAcyclic} (\ref{it:UPAcyclic2}) it just remains to show that the map $g_* : \pi_1({\nabla}_\partial) \to \pi_1(E_k)$ factors over $\pi_1({\nabla}_\partial) \to \pi_1(C_k(C))$, with any basepoint. To see this it will suffice to show that $(\pi_k)_* : \pi_1(E_k) \to \pi_1(C_k(C))$ is an isomorphism, with any basepoint, so it suffices to show that the fibres of $\pi_k : E_k \to C_k(C)$ are 1-connected: we will show that they are in fact $(d-2)$-connected.

The fibre of $\pi_k : E_k \to C_k(C)$ over a configuration $\xi$ is given by the total homotopy fibre $F_\xi := \tohofib_{S \subset \xi} \mathrm{Emb}(-,M)$ of the $k$-cube
$$(S \subset \xi) \longmapsto \mathrm{Emb}(S,M).$$
To form the total homotopy fibre, we must be given compatible basepoints at each corner of the cube except the initial one: these are given by applying the ``partial embedding'' $e \in T_{k-1} \mathrm{Emb}_{B}(C, M)$ to the proper subsets of $\xi$. We show that the total homotopy fibres $F_\xi$ are $(d-2)$-connected by induction over $k$. For $k=2$ there is a diagram of fibration sequences
\begin{equation*}
  \begin{tikzcd}
F_{\{1,2\}} \rar &  M \setminus \{*\} \dar \rar & M \dar\\
 & \mathrm{Emb}(\{1,2\},M) \rar \dar & \mathrm{Emb}(\{1\},M) \dar\\
 & \mathrm{Emb}(\{2\},M) \rar & *
  \end{tikzcd}
\end{equation*}
and as $M$ has non-empty boundary, the inclusion $M \setminus \{*\} \to M$ is split up to isotopy. As $M \setminus \{*\} \simeq M \vee S^{d-1}$, the total homotopy fibre $F_{\{1,2\}}$ is $(d-2)$-connected.  More generally \cite[Proposition 5.5.4]{munsonvolic}, a $k$-cube can be regarded as a map of $(k-1)$-cubes, and the total homotopy fibre of the $k$-cube can be computed as the homotopy fibre of the induced maps of total homotopy fibres of the $(k-1)$-cubes.  In our situation, this gives a fibration sequence
$$\tohofib_{\mathclap{S \subset \{1,2,\ldots,k\}}} \mathrm{Emb}(-,M) \lra \tohofib_{\mathclap{S \subset \{1,2,\ldots,k-1\}}} \mathrm{Emb}(-,M\setminus \{*\}) \lra \tohofib_{\mathclap{S \subset \{1,2,\ldots,k-1\}}} \mathrm{Emb}(-,M)$$
which has a section, so the left-hand term is at least as connected as the middle term: by induction over $k$ is is therefore $(d-2)$-connected.
\end{proof}

\begin{lemma}\label{lem:acyclic}
Let $C$ be a Hausdorff topological space and $B \subset C$ a closed subspace such that the inclusion $B \hookrightarrow C$ is an acyclic map, and define $D_k \subset C^k$ by the formula~\eqref{eq:1}.  Then the inclusion $D_k \hookrightarrow C^k$ is acyclic.
\end{lemma}

\begin{remark}
The lemma is true in the stated generality, provided ``acyclic'' is taken in the sense of \v{C}ech cohomology: a map $f: X \to Y$ is acyclic when $f^*: \check{H}^*(Y;\mathcal{A}) \to \check{H}^*(X;f^* \mathcal{A})$ is an isomorphism for all local systems $\mathcal{A}$ on $Y$.  In our application, all spaces are locally contractible, $D_k$ because it is a retract of an open subset of the manifold $C^k$, so \v{C}ech cohomology agrees with singular cohomology and acyclicity agrees with the standard notion, in turn equivalent to homotopy fibres having the singular homology of a point.

Working in this generality is convenient for the proof, for example circumventing whether the pushout diagrams below are also homotopy pushouts.
\end{remark}

\begin{proof}[Proof of Lemma \ref{lem:acyclic}]
  We use induction on $k$, the case $k=1$ being tautological.  For the induction step, we use two filtrations.  Let $F_j C^k \subset C^k$ be the closed subspace consisting of those $y = (y_1, \dots, y_k)$ for which $\{i \mid y_i \in C \setminus B\}$ has cardinality at most $j$.  Then $F_0 C^k = B^k$, and for $j \geq 1$ there is a commutative diagram
  \begin{equation}\label{eq:3}
    \begin{tikzcd}
      \displaystyle\coprod_{\mathclap{\sigma \in \mathfrak{S}_{k-j,j}}} B^{k-j} \times F_{j-1}C^j \rar[hookrightarrow] \dar &
      \displaystyle\coprod_{\mathclap{\sigma \in \mathfrak{S}_{k-j,j}}} B^{k-j} \times C^j \dar[two heads] \\
      F_{j-1} C^k \rar & F_j C^k,
    \end{tikzcd}
  \end{equation}
  where $\mathfrak{S}_{k-j,j} \subset \mathfrak{S}_k$ denotes the set of $(k-j,j)$-shuffles and the left vertical map is defined by
  \begin{align*}
    B^{k-j} \times C^j & \lra C^k\\
    (y_1, \dots, y_{k-j}, y_{k-j+1}, \dots, y_k) & \longmapsto (y_{\sigma^{-1}(1)}, \dots, y_{\sigma^{-1}(k-j)}, y_{\sigma^{-1}(k-j+1)}, \dots, y_{\sigma^{-1}(k)}).
  \end{align*}
  A point $y \in F_j C^k \setminus F_{j-1} C^k$ is in the image of this injective map for precisely one $\sigma \in \mathfrak{S}_{k-j,j}$, from which we deduce that the square diagram of sets underlying diagram~(\ref{eq:3}) is pushout.  We claim that~(\ref{eq:3}) is pushout in topological spaces, too.  To see this, we consider the composition
  \begin{equation*}
    F_{j-1} C^k \amalg \coprod_{\mathfrak{S}_{k-j,j}} B^{k-j} \times C^j \hookrightarrow C^k \amalg \coprod_{\mathfrak{S}_{k-j,j}} C^k \lra C^k,
  \end{equation*}
  whose image is $F_j C^k$.  The composition is a closed map, since $\mathfrak{S}_{k-j,j}$ is finite, and therefore the quotient topology and the subspace topology on the image $F_j C^k \subset C^k$ agree.

  We now claim that the inclusion $F_{k-1} C^k \hookrightarrow C^k$ is acyclic.  By induction on $k$, we may assume that the inclusions $F_{j-1} C^j \to C^j$ are acyclic for all $j < k$.  It then follows from the pushout diagram that $F_{j-1} C^k \to F_j C^k$ is also an acyclic map.  In the diagram
  \begin{equation*}
    F_0 C^k \hookrightarrow F_1 C^k \hookrightarrow \dots \hookrightarrow F_{k-1} C^k \hookrightarrow C^k,
  \end{equation*}
  the composition is acyclic and all arrows except possibly the last one are acyclic.  Then the last inclusion $F_{k-1} C^k \hookrightarrow C^k$ must be acyclic too, completing the induction step.

  Next we consider a filtration $F'$, defined by 
  \begin{equation*}
    F'_j C^k = F_{k-1} C^k \cup \{y \in C^k \mid \text{the cardinality of $\{y_1, \dots, y_k\}$ is at most $j$}\}.
  \end{equation*}
  In this case we have $F'_0 C^k = F_{k-1} C^k$ and pushout diagrams
  \begin{equation*}
    \begin{tikzcd}
      \displaystyle\coprod_\theta  F'_{j-1} C^j \rar[hookrightarrow]\dar &
      \displaystyle\coprod_\theta C^j\dar[two heads]\\
      F'_{j-1} C^k \rar & F'_j C^k,
    \end{tikzcd}
  \end{equation*}
  where the coproduct is indexed by surjections $\theta: \{1, \dots, k\} \to \{1, \dots, j\}$, one in each $\mathfrak{S}_j$-orbit, and the vertical maps are defined by
  \begin{equation*}
    (y_1, \dots, y_j) \longmapsto (y_{\theta(1)}, \dots, y_{\theta(j)}).
  \end{equation*}
  The pushout diagram expresses that a point $y \in F'_j C^k \setminus F'_{j-1} C^k$ is in the image of this injective map for some $\theta$, uniquely so up to permuting coordinates in the domain (viz., choose an arbitrary bijection $\{y_1, \dots, y_k\} \to \{1, \dots, j\}$ and compose with the surjection $i \mapsto y_i$).

By induction on $k$, the top horizontal map is acyclic for all $j < k$, so we deduce that $F'_{j-1} C^k \to F'_j C^k$ is also acyclic.  Since $F'_0 C^k \to C^k$ is acyclic, we deduce that the inclusion  $D_k = F'_{k-1} C^k \to C^k$ is acyclic, as claimed.
\end{proof}

The discussion so far assumed only that $B \to C$ is acyclic and showed that $T_\infty \mathrm{Emb}_B(C, M) \simeq *$. We now claim that when in addition $B' \to C$ is a homotopy equivalence, embedding calculus converges, i.e.\ the map $\mathrm{Emb}_B(C, M) \to T_\infty \mathrm{Emb}_B(C, M)$ is an equivalence, which proves Theorem \ref{thm:B}. By \cite[Corollary 2.5]{EmbImmII}, for convergence it suffices to show that $C$ admits a handle structure relative to $B$ with all handles of index $\leq d-3$, or equivalently that $C$ admits a handle structure relative to $B'$ with all handles of index $\geq 3$. Using that $d \geq 6$ this follows from the proof of the $h$-cobordism theorem: starting from any handle structure on $C$ relative to $B'$, eliminate low-dimensional handles using that $B' \to C$ is an equivalence.

\begin{remark}\label{rem:handles}
For a one-sided $h$-cobordism $C : B \leadsto B'$ of dimension $d \geq 6$, the proof of the $h$-cobordism theorem will succeed in eliminating handles of index $\leq d-3$ relative to $B'$, and so shows that $C$ admits a handle structure relative to $B$ only having handles of index $\leq 3$. On the other hand, as $B \to C$ is 1-connected it is simple to remove any 0- and 1-handles, at the expense of perhaps creating more 2-handles. Thus any one-sided $h$-cobordism $C$ may be obtained from $B$ up to diffeomorphism by attaching only 2- and 3-handles.
\end{remark}

\section{Second proof of Theorem \ref{thm:B}: Resolution by 2-handles}

Before explaining the strategy, we discuss a certain standard form of one-sided $h$-cobordisms $C : B \leadsto B'$ dimension $d \geq 6$, and use it to reduce to a special case.

\begin{lemma}
If $B$ has dimension $d-1 \geq 5$ then there exists a closed $(d-2)$-manifold $Y$, an embedding $[-\infty,\infty] \times Y \hookrightarrow B$, and a handle structure on $B$ relative to $[-\infty,\infty] \times Y$ with handles of index $\geq 3$ only; in particular the embedding is 2-connected.
\end{lemma}
\begin{proof}
  Let us write $\delta = d-1$ for the dimension of $B$.  Let $f : B \to [0,\delta + \tfrac{1}{2}]$ be a proper and self-indexing Morse function with $f^{-1}(\delta+\tfrac{1}{2}) = \partial B$, and set $Y := f^{-1}(2 + \tfrac{1}{2})$. As $2 + \tfrac{1}{2}$ is a regular value of $f$, $Y$ is a closed $(\delta-1)$-manifold and has a neighbourhood diffeomorphic to $[-\infty,\infty] \times Y$. Now $B$ is obtained from $[-\infty,\infty] \times Y$ by handles of index $\geq \delta-2$ from below and handles of index $\delta \geq 3$ from above. As $\delta \geq 5$, it follows that $B$ is obtained up to homotopy from $[-\infty,\infty] \times Y$ by attaching cells of dimension $\geq 3$, so the inclusion $[-\infty, \infty] \times Y \hookrightarrow B$ is 2-connected.
\end{proof}

\begin{lemma}
  Let $[-\infty,\infty] \times Y \hookrightarrow B$ be an embedding as in the previous lemma, and write $B_2 \subset B$ for the image of $[-\infty,0] \times Y$.  Then there exists a cobordism $C_2 : B_2 \leadsto B'_2$ restricting to a trivial cobordism $\partial B_2 \leadsto \partial B'_2$ on the boundary, and a diffeomorphism
  \begin{equation}\label{eq:2}
   C \xrightarrow{\cong} C_2 \cup ([0,1] \times B \setminus \mathrm{Int}(B_2))
  \end{equation}
  relative to $B$.
\end{lemma}
\begin{proof}
  As in Remark \ref{rem:handles}, the cobordism $C : B \leadsto B'$ admits a handle structure relative to $B$ only having 2- and 3-handles.  As in the proof of the $h$-cobordism theorem, $C$ is diffeomorphic to a composition $C' \circ C''$, where $C'': B \leadsto B''$ is the trace of surgery along disjoint embeddings $S^1 \times D^{d-2} \to B$ and $C': B'' \to B'$ is the trace of surgery along disjoint embeddings $S^2 \times D^{d-3} \to B''$.

  The attaching maps $S^1 \times D^{d-2} \hookrightarrow B$ of the 2-handles may be isotoped to have image in $B_2$: first use transversality to avoid the co-cores of the handles of $B$ relative to $B_2 \subset B$, which have codimension at least 3, then retract the complement of those co-cores to $\mathrm{Int}(B_2)$ by an isotopy.  The cobordism $C'': B \leadsto B''$ is then diffeomorphic to a gluing of the form
  \begin{equation*}
    C_2'' \cup ([0,1] \times B \setminus \mathrm{Int}(B_2)),
  \end{equation*}
  where $C''_2$ is the trace of surgery along the isotoped embeddings $S^1 \times D^{d-2} \to B_2$.  In particular, the complement $B \setminus \mathrm{Int}(B_2)$ also sits inside $B''$, and we may apply the same argument to the attaching maps $S^2 \times D^{d-3} \to B''$ to isotope the attaching maps into the outgoing boundary of $C_2''$.  Letting $C'_2$ be the trace of surgery on those embeddings and $C_2 = C'_2 \circ C''_2$ gives the required decomposition~(\ref{eq:2}) of $C$.
\end{proof}

\begin{corollary}\label{corollary:product-B-suffices}
  It suffices to prove Theorem~\ref{thm:B} in the case $B = [-\infty,0] \times Y$ for some closed $(d-2)$-manifold $Y$, and $e_0$ is the restriction of an embedding $[-\infty,\infty] \times Y \hookrightarrow \partial M$.
\end{corollary}
\begin{proof}
  For arbitrary $e_0 : B \hookrightarrow \partial M$ as in the statement of the Theorem, we may choose a decomposition~(\ref{eq:2}) as above.  Then the restriction map
  $$\mathrm{Emb}_B(C, M) \lra \mathrm{Emb}_{B_2}(C_2, M)$$
  is a homotopy equivalence, as the homotopy fibre is a space of embeddings of a collar.  If the theorem holds for the one-sided $h$-cobordism $C_2: B_2 \leadsto B'_2$, it therefore also holds for $C: B \leadsto B'$.
\end{proof}

In the rest of this section, we shall assume $B\hookrightarrow \partial M$ is given a product decomposition as in the Corollary above.  For notational convenience, we may also arrange that the basepoint $b \in B$ for fundamental groups is in the image of $\{0\} \times Y \subset B \hookrightarrow \partial M$.

\begin{construction}\label{const:T}
Choose an embedding $\{1,2,\ldots, r\} \times S^1 \hookrightarrow (0,1) \times Y$ of loops whose conjugacy classes generate the kernel of 
$$\pi_1([0,1) \times Y,b) \cong \pi_1(\{0\} \times Y,b) \overset{\sim}\lra \pi_1(B,b) \overset{(f_B)_*}\lra \pi.$$
By assumption these loops are nullhomotopic under $(0,1) \times Y \subset \partial M \subset M$, so by general position they can be extended to a smooth embedding $\{1,2,\ldots, r\} \times D^2 \hookrightarrow M$. As $D^2$ is contractible, over each path component this embedding has trivial normal bundle, giving an embedding $\{1,2,\ldots, r\} \times D^2 \times D^{d-2} \hookrightarrow M$ and so over the boundary an embedding
$$\phi : \{1,2,\ldots, r\} \times S^1 \times D^{d-2} \hookrightarrow (0,1) \times Y.$$
Writing 
$$T : [0,1] \times Y \leadsto ([0,1] \times Y)'$$
for the trace of the surgery on $\phi$, we have arranged that $T$ embeds in $M$ relative to $[0,1] \times Y$.
\end{construction}

We write $[p] := \{0,1,\ldots, p\}$ considered as a totally ordered set; the category $\Delta_{\inj}$ with objects the $[p]$ for $p \geq 0$ and morphisms the order-preserving injections is the {semi-simplicial category}, and a functor $\Delta^{op}_{\inj} \to \mathsf{Top}$ is called a \emph{semi-simplicial space}. The {augmented semi-simplicial category} $\Delta_{\aug}$ is obtained by adjoining the totally ordered set $[-1] := \emptyset$, and an \emph{augmented semi-simplicial space} is a functor $\Delta^{op}_{\aug} \to \mathsf{Top}$.

\begin{definition}\label{def:Xbullet}
  For $C: B \leadsto B'$ as above (i.e., as in Corollary~\ref{corollary:product-B-suffices}), let $X_\bullet$ denote the augmented semi-simplicial space with $X_p$ given by the space of collared embeddings
$$e : C \amalg ([p] \times T) \lra M$$
of the disjoint union of $C$ and $p+1$ copies of $T$ such that
\begin{enumerate}[(i)]
\item $e\vert_{B} = \mathrm{inc} : B \to \partial M$,

\item $e\vert_{\{i\} \times [0,1] \times Y}$ sends $\{i\} \times [0,1] \times Y$ into $[0,\infty] \times Y \subset \partial M$ by a map of the form $a_i \times \mathrm{id}_Y$ for $a_i : x \mapsto x+t_i : [0,1] \to [0,\infty]$ a translation, such that $0 \leq t_0 < t_1 < \cdots < t_p$.

\end{enumerate}
Elements of $X_1$ may be depicted as in Figure~\ref{fig:1}.
We topologise $X_p$ as a subspace of the space of all collared embeddings. The face map associated to $\theta : [p] \to [q] \in \Delta_{\aug}$ is induced by precomposing with 
$$\mathrm{id}_C \amalg \theta \times T : C \amalg ([p] \times T) \lra C \amalg ([q] \times T).$$
It has $X_{-1} = \mathrm{Emb}_{B}(C, M)$.

Similarly, let $Z_\bullet$ denote the augmented semi-simplicial space with $p$-simplices given by the space of collared embeddings
$$e : [p] \times T \lra M$$
such that (ii) holds, with the analogous face maps. It has $Z_{-1} = \Emb(\emptyset, M)$, which is a single point.
\end{definition}

\begin{figure}[h]
\includegraphics{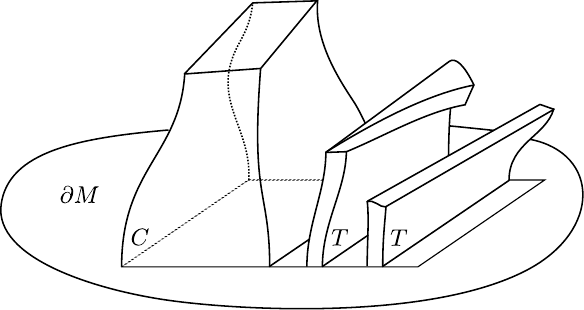}
\caption{A typical 1-simplex in $X_\bullet$.}
\label{fig:1}
\end{figure}

Restricting an embedding of $C \amalg ([p] \times T)$ to $[p] \times T$ defines a map $f_\bullet : X_\bullet \to Z_\bullet$ of augmented semi-simplicial spaces. This gives a commutative square
\begin{equation}\label{eq:square}
\begin{tikzcd}
& {|X_\bullet|} \arrow[r, "{|f_\bullet|}"] \dar & {|Z_\bullet|} \dar\\
\mathrm{Emb}_{B}(C, M) \arrow[r, equals] & X_{-1} \arrow[r, "f_{-1}"] & Z_{-1} \arrow[r, equals] & *
\end{tikzcd}
\end{equation}
so to show that $\mathrm{Emb}_{B}(C, M)$ is contractible, i.e.\ that $f_{-1}$ is an equivalence, it suffices to show that the other three maps are equivalences.

The following lemma, together with homotopy invariance of geometric realisation of semi-simplicial spaces (see e.g.\ \cite[Theorem 2.2]{ER-WSx}), shows that the top map in \eqref{eq:square} is an equivalence.

\begin{lemma}
The maps $f_p : X_p \to Z_p$ are equivalences for all $p \geq 0$.
\end{lemma}
The proof of this Lemma is reminiscent of the arguments of Sections 4 and 5 of \cite{grwstab2}.
\begin{proof}
The map $f_p$ is a Serre fibration, by the parameterised isotopy extension theorem. Its fibre over an embedding $e : [p] \times T \to M$ is the space of embeddings of $C$ into $M$ relative to $B$ which have image disjoint from $\mathrm{Im}(e)$. Equivalently, it is the space of embeddings of $C$ into
$M \setminus \mathrm{Im}(e)$ relative to $B = [-\infty,0] \times Y$. The embedding $e$ can be extended to an embedding 
$$\bar{e} : ([p] \times T) \cup ([0,1] \times ([0,\infty] \times Y \setminus \mathrm{int}(e([p] \times [0,1] \times Y)))) \lra M,$$
and because the domains of these embeddings differ by a collar, the fibre of $f_p$ is also equivalent to $\mathrm{Emb}_B(C, M \setminus \mathrm{Im}(\bar{e}))$.  This situation may be depicted as in Figure~\ref{fig:2}.

\begin{figure}[h]
\includegraphics{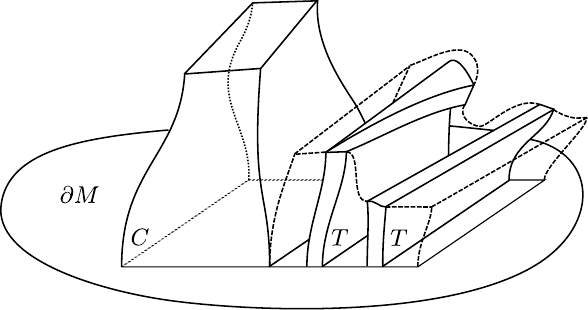}
\caption{The image of an embedding $\bar{e}$ extending $e$ is indicated with dashed lines.}
\label{fig:2}
\end{figure}

We may replace the target in this space of embeddings with the manifold $\overline{M}$ formed by rounding corners of $M \setminus \mathrm{int}(\mathrm{Im}(\bar{e}))$.
Then letting
$$\overline{C} := C \cup \{\text{a collar of } \partial \overline{M} \subset \overline{M}\},$$
which is again a one-sided $h$-cobordism, we see that the restriction map
\begin{equation}\label{eq:SurgeredCollar}
\mathrm{Emb}_{\partial \overline{M}}(\overline{C}, \overline{M}) \lra \mathrm{Emb}_{B}(C, \overline{M})
\end{equation}
is again an equivalence, because $C$ and $\overline{C}$ differ by a collar. But the one-sided $h$-cobordism $\overline{C}$ is in fact an $h$-cobordism, because having removed $\mathrm{int}(\mathrm{Im}(\bar{e}))$ from $M$ means that $\partial \overline{M}$ differs from $\partial M$ by doing at least one collection of surgeries $\phi$ inside $[0,\infty] \times Y$, so the 2-handles of $\overline{C}$ are attached to $\partial \overline{M}$ along nullhomotopic maps.
Being an $h$-cobordism, the domain of \eqref{eq:SurgeredCollar} is contractible, as required.
\end{proof}

The following lemma shows that the vertical maps in \eqref{eq:square} are equivalences.

\begin{lemma}
The augmentation maps
$$|X_\bullet| \to X_{-1} \quad\quad\quad |Z_\bullet| \to Z_{-1}$$
are equivalences.
\end{lemma}
\begin{proof}
We will treat $X_\bullet \to X_{-1}$, then $Z_\bullet \to Z_{-1}$ is done in exactly the same way but with some simplifications. We first define a more relaxed version $X'_\bullet$ of $X_\bullet$, by letting a $p$-simplex be a map
$$f : C \amalg ([p] \times T) \lra M$$
which is an embedding when restricted to $C$ and to each $\{i\} \times T$, and such that $f(C)$ is disjoint from each $f(\{i\} \times T)$, but where we only insist that the $f(\{i\} \times T)$ have disjoint cores to each other. Here, the \emph{core} of $T$ means the subspace
$$([0,1] \times Y) \cup (\{1,2,\ldots, r\} \times D^2) \subset T$$
given by its incoming boundary and the cores of the 2-handles. Apart from this difference we continue to insist that properties (i) and (ii) of Definition \ref{def:Xbullet} hold. There is an inclusion $X_\bullet \to X'_\bullet$, and the maps $X_p \to X'_p$ are all equivalences, by choosing a 1-parameter family of isotopies of $T$ which shrink it down into arbitrarily small neighbourhoods of its core (see e.g.\ \cite[Proposition 6.7]{GRWActa} for a similar argument). By \cite[Theorem 2.2]{ER-WSx} the map $|X_\bullet| \to |X'_\bullet|$ is an equivalence, so we are reduced to showing that $|X'_\bullet| \to X_{-1}$ is an equivalence.

To do this, we will show that $X'_\bullet \to X_{-1}$ is an augmented topological flag complex  \cite[Definition 6.1]{GRWActa}, and then verify the hypotheses of \cite[Theorem 6.2]{GRWActa}. To see that it is an augmented topological flag complex, first note the $(p+1)$-fold fibre product $X'_0 \times_{X_{-1}} X'_0 \times_{X_{-1}} \cdots \times_{X_{-1}} X'_0$ consists of an embedding $e_C$ of $C$ into $M$ satisfying (i) along with $(p+1)$ embeddings $e_T^i$ of $T$ individually satisfying (ii), which are each disjoint from $C$, so that $X'_p$ is indeed the open subspace of those tuples where the $e_T^i$'s in addition have disjoint cores from from each other. Furthermore, a collection of $e_T^i$'s have disjoint cores from from each other if and only if they pairwise do. Thus $X'_\bullet \to X_{-1}$ is indeed an augmented topological flag complex.

We now verify the hypotheses of \cite[Theorem 6.2]{GRWActa}. The restriction map $X'_0 \to X_{-1}$ is a Serre fibration by the parameterised isotopy extension theorem, so hypothesis (i) is satisfied. 

To verify hypothesis (ii) of \cite[Theorem 6.2]{GRWActa} we must show that for any embedding $e_C : C \hookrightarrow M$ relative to $B$ there exists an embedding $e_T : T \hookrightarrow M$ relative to $[0,1] \times Y$ and disjoint from $e_C(C)$. By Construction \ref{const:T} we have arranged that there is an embedding $f : T \hookrightarrow M$ relative to $[0,1] \times Y$, which is then disjoint from $C$ on the boundary. As $T$ only has 2-handles relative to $\partial M$, $C$ only has 2- and 3-handles relative to $\partial M$, and $d \geq 6$, by general position the cores of handles of $f(T)$ will be disjoint from those of $e_C(C)$, and $f$ can therefore be isotoped relative to $\partial M$ to miss $e_C(C)$.

To verify hypothesis (iii) of \cite[Theorem 6.2]{GRWActa} we must show that given an embedding $e_C : C \hookrightarrow M$ relative to $B$, and a non-empty set of embeddings $\{e_T^i : T \hookrightarrow M\}$ which are disjoint from $e_C(C)$, there is another $e_T : T \hookrightarrow M$ which misses $e_C(C)$ and whose core is disjoint from the cores of all the $e_T^i(T)$'s. If on the boundary $e_T^i$ is given by $(x, y) \mapsto (x + t_i, y) : [0,1] \times Y \to [0,\infty] \times Y$, then choosing some $t \gg t_i$ there by Construction \ref{const:T} is an embedding $f : T \hookrightarrow M$ which is given by $(x, y) \mapsto (x + t, y)$ on the boundary. This is then disjoint from $e_C(C)$ and the $e_T^i(T)$'s on the boundary. As in the paragraph above we may isotope $f$ to be disjoint from $C$. Finally, after perhaps perturbing $f$ the cores of $f(T)$ will be disjoint from those of all $e_T^i(T)$'s, because they are all 2-dimensional and $d \geq 6$.

Thus \cite[Theorem 6.2]{GRWActa} applies, showing that $|X_\bullet| \simeq |X'_\bullet | \to X_{-1}$ is an equivalence.  The case $|Z_\bullet| \to Z_{-1}$ is similar but easier.
\end{proof}

\section{Consequences for embedding spaces and embedding calculus}\label{sec:Consequences}

Theorem \ref{thm:B} can be used to show that spaces of embeddings of manifolds which differ by a one-sided $h$-cobordism are homotopy equivalent. We give the following example. 

\begin{corollary}\label{cor:C}
Let $N \subset N'$ be compact smooth manifolds of dimension $d \geq 6$ both admitting handle structures with all handles of index $\leq d-3$, and such that the inclusion is acyclic. Then for any $d$-manifold $M$ the restriction map
$$\mathrm{res}: \mathrm{Emb}(N', M) \lra \mathrm{Emb}(N, M)$$
is a homotopy equivalence onto the subspace of those embeddings $e$ such that $e_* : \pi_1(N) \to \pi_1(M)$ factors through $\pi_1(N) \to \pi_1(N')$.
\end{corollary}
The condition about fundamental groups is required to hold for all basepoints $b \in N$, or equivalently, one in each path component.
\begin{proof}
  Without loss of generality, we may assume that $N' = N \cup_{\partial N} C$ for some cobordism $C : \partial N \leadsto \partial N'$.  The condition on handles implies that the inclusions $\partial N' \to C \to N'$ and $\partial N \to N$ induce bijections on $\pi_0$, and for convenience of notation we shall assume these manifolds are all path connected.  Choosing a basepoint $b \in \partial N'$ and letting $\pi = \pi_1(N',b)$, the condition on handles of $N$ also implies that $\pi_1(C,b) \overset{\sim}\to \pi_1(N',b) = \pi$, and under the condition on handles of $N'$ we have $\pi_1(\partial N',b) \overset{\sim}\to \pi_1(N',b) = \pi$. Then $H_*(C, \partial N ; \bZ[\pi]) = H_*(N', N ; \bZ[\pi])=0$ by excision, so $C$ is a one-sided $h$-cobordism on $\partial N$ over $B\pi$. By the isotopy extension theorem, the homotopy fibre of $\mathrm{res}$ over a point $e : N \hookrightarrow M$ is given by $\mathrm{Emb}_{\partial N}(C, M \setminus \mathrm{int}(e(N)))$. If $e$ satisfies the given condition on fundamental groups, then this is contractible by Theorem \ref{thm:B}.
\end{proof}

Because of the assumption about handle dimension of the two manifolds, in the setting of Corollary \ref{cor:C} embedding calculus converges, and so the conclusion can also be phrased in terms of $T_\infty \mathrm{Emb}(-, M)$. 

In the setting of the following Corollary embedding calculus need not converge, so the result can only be formulated for $T_\infty \mathrm{Emb}(-, M)$. Recall that via the maps
$$T_\infty \mathrm{Emb}(N, M) \lra T_1 \mathrm{Emb}(N, M) \overset{\eqref{eq:ImmToMaps}}\lra \map(N, M)$$
each $T_\infty e \in T_\infty \mathrm{Emb}(N, M)$ determines an underlying continuous map $e : N \to M$.

\begin{corollary}\label{cor:D}
  Let $N \subset N'$ be compact smooth manifolds
such that either:
\begin{enumerate}[(i)]
\item\label{it:corD:1}  $N$ admits a handle structure with all handles of index $\leq d-3$, and the inclusion $N \subset N'$ is acyclic, or
\item\label{it:corD:2} the inclusion $N \subset N'$ is a homotopy equivalence.

\end{enumerate}

Then the restriction map
$$\mathrm{res}: T_\infty\mathrm{Emb}(N', M) \lra T_\infty\mathrm{Emb}(N, M)$$
is a homotopy equivalence onto the path-components of those $T_\infty e$ such that $e_* : \pi_1(N) \to \pi_1(M)$ factors through $\pi_1(N) \to \pi_1(N')$. (This condition is vacuous in case (ii).)
\end{corollary}
\begin{proof}
In case (\ref{it:corD:1}) we repeat the proof of Corollary \ref{cor:C}, using the isotopy extension theorem for $T_\infty\mathrm{Emb}(-, M)$ of Knudsen--Kupers \cite[Theorem 6.1]{KnudsenKupers}. Under our assumption on the handle dimension of $N$ we have $\mathrm{Emb}(N, M) \overset{\sim}\to T_\infty\mathrm{Emb}(N, M)$, and for any $e \in \mathrm{Emb}(N, M)$ this isotopy extension theorem gives a fibre sequence
$$T_\infty \mathrm{Emb}_{\partial N}(C, M \setminus e(\mathrm{int}(N))) \lra T_\infty\mathrm{Emb}(N', M) \overset{\mathrm{res}}\lra T_\infty\mathrm{Emb}(N, M),$$
where $C : \partial N \leadsto \partial N'$ is as in the proof of Corollary \ref{cor:C}. As $N$ has a handle structure with all handles of index $\leq d-3$, it can be constructed from $\partial N$ by attaching handles of index $\geq 3$. Thus the inclusion $C \to N'$ induces an isomorphism on $\pi_1$, so all local coefficient systems on $C$ extend to $N'$. It then follows by excision that the inclusion $\partial N \to C$ is acyclic, and the argument of Section 2 then shows that $T_\infty \mathrm{Emb}_{\partial N}(C, M \setminus e(\mathrm{int}(N)))$ is contractible.

In case (\ref{it:corD:2}) we follow Boavida de Brito--Weiss \cite{weisspedrosheaves} (see also \cite{KKDisc}) and formulate embedding calculus in terms of the presheaves $E_{N}(-) := \mathrm{Emb}(-, N)$ on the category $\mathsf{Disc}_d$ of finite disjoint unions of $d$-dimensional Euclidean spaces and embeddings between them. In this formulation $T_\infty \mathrm{Emb}(-, M)$ is the mapping space $\mathrm{Map}_{\mathsf{PSh}(\mathsf{Disc}_d)}(E_{-}, E_M)$ in the $\infty$-category of $\mathsf{Disc}_d$-presheaves, so it suffices to show that the induced map $E_N \to E_{N'}$ is an equivalence of presheaves. That is, that the map
  \begin{equation*}
    E_N(U) = \mathrm{Emb}(U, N) \lra \mathrm{Emb}(U, N') = E_{N'}(U)
  \end{equation*}
  is an equivalence for all $U \in \mathsf{Disc}_d$.  This may be seen by induction on the number of balls in $U$ as follows. If $U = \bR^d \amalg V$ then we can form the commutative diagram
\begin{equation*}
  \begin{tikzcd}
\mathrm{Emb}(\bR^d \amalg V, N) \rar \dar & \mathrm{Emb}(\bR^d \amalg V, N') \dar\\
\mathrm{Emb}(D^d, N) \rar & \mathrm{Emb}(D^d, N')
  \end{tikzcd}
\end{equation*}
where the vertical maps are given by restriction to $D^d \subset \bR^d$, and are Serre fibrations by the isotopy extension theorem. The bottom map may be identified with the map between frame bundles $\mathrm{Fr}(N) \to \mathrm{Fr}(N')$, which is an equivalence, so it remains to show that the map between fibres is an equivalence. The vertical fibre over an $e : D^d \hookrightarrow N$ can be identified with
$$\mathrm{Emb}_{\partial e(D^d)}(\bR^d \setminus \mathrm{int}(e(D^d)) \amalg V, N \setminus \mathrm{int}(e(D^d))) \overset{\sim}\lra \mathrm{Emb}(V, N \setminus e(D^d)),$$
where the equivalence is by the contractibility of the space of (open) collars, and similarly for $N'$. As $N \hookrightarrow N'$ is a homotopy equivalence so is $N \setminus e(D^d) \hookrightarrow N' \setminus e(D^d)$, and hence $\mathrm{Emb}(V, N \setminus e(D^d)) \to \mathrm{Emb}(V, N' \setminus e(D^d))$ is an equivalence by induction as $V$ consists of fewer balls than $U$. 
\end{proof}

The argument for Corollary \ref{cor:D} (\ref{it:corD:2}) we have given means it is not actually a corollary of our earlier results: we are grateful to a referee for suggesting this argument, which replaces a more complicated one which used Theorem \ref{thm:B}. As an application of this result, we give the following concrete example of non-convergence of embedding calculus.

\begin{example}
Let $\Delta$ be a contractible smooth manifold of dimension $d \geq 6$. The double $\Delta \cup_{\partial \Delta} \Delta$ is simply-connected and so is a homotopy sphere. It is the boundary of the contractible manifold $\Delta \times [0,1]$, and removing a disc from the interior of this manifold shows that $\Delta \cup_{\partial \Delta} \Delta$ is $h$-cobordant to $S^d$, so is diffeomorphic to it. In particular, it follows that there is a smooth embedding $f : \Delta \hookrightarrow D^d$. This induces an orientation of $\Delta$, and there is also an oriented embedding $e : D^d \hookrightarrow \Delta$ given by a coordinate chart. 

If $\partial \Delta$ is not simply-connected, then the composition $e \circ f : \Delta \hookrightarrow \Delta$ is \emph{not} isotopic to the identity: if it were then its complement would be a cobordism $\partial \Delta \leadsto \partial\Delta$ diffeomorphic to a cylinder, but by construction it factors as $\partial \Delta \leadsto S^{d-1} \leadsto \partial \Delta$ and the second cobordism is simply-connected so this cannot be diffeomorphic to a cylinder.

By Corollary \ref{cor:D} (\ref{it:corD:2}) the restriction map
$$T_\infty \mathrm{Emb}(\Delta, \Delta) \lra T_\infty \mathrm{Emb}(D^d, \Delta) \simeq \mathrm{Emb}(D^d, \Delta)$$
is an equivalence. The elements $e \circ f, \mathrm{id}_\Delta \in  \mathrm{Emb}(\Delta, \Delta)$ lie in the same path component when restricted along $e : D^d \to \Delta$, as $f \circ e$ is isotopic to the identity, so it follows that $T_\infty(e \circ f)$ and $T_\infty(\mathrm{id})$ lie in the same path-component of $T_\infty \mathrm{Emb}(\Delta, \Delta)$. Thus $\mathrm{Emb}(\Delta, \Delta) \to T_\infty \mathrm{Emb}(\Delta, \Delta)$ is not $\pi_0$-injective.
\end{example}

\bibliographystyle{amsalpha}
\bibliography{biblio}

\end{document}